\documentclass[10pt,reqno]{amsart}
\usepackage{amsmath}
\usepackage{amsthm}
\usepackage{amssymb}
\usepackage{amsfonts}
\usepackage{latexsym}
\usepackage{cite, hyperref}
\usepackage[usenames]{color}
\usepackage{enumitem}

\usepackage[sc]{mathpazo}
\usepackage[T1]{fontenc}

\renewcommand{\phi}{\varphi}

\newcommand{\C}{\mathbb{C}}

\newcommand{\SSS}{\mathcal{S}}
\newcommand{\D}{\mathbb{D}}

\newcommand{\T}{\mathbb{T}}

\renewcommand{\vec}[1]{{\bf #1}}

\newcommand{\vecspan}{\operatorname{span}}

\newcommand{\nullity}{\operatorname{nullity}}

\newcommand{\minimatrix}[4]{\begin{bmatrix} #1 & #2 \\ #3 & #4 \end{bmatrix}  }

\newcommand{\rank}{\operatorname{rank}}

\newtheorem{theorem}{Theorem}
\newtheorem{lemma}[theorem]{Lemma}
\newtheorem{proposition}[theorem]{Proposition}

\theoremstyle{definition}

\theoremstyle{remark}

\allowdisplaybreaks
\begin{document}
	\title{Matrices similar to partial isometries}
	\date{\today}

	\author{Stephan Ramon Garcia}
	\address{Department of Mathematics, Pomona College, 610 N. College Ave., Claremont, CA 91711} 
	\email{stephan.garcia@pomona.edu}
	\urladdr{\url{http://pages.pomona.edu/~sg064747}}
	
	\author{David Sherman}
	\address{Department of Mathematics, University of Virginia, P.O. Box 400137, Charlottesville, VA 22904-4137}
	\email{dsherman@virginia.edu}
	\urladdr{\url{http://people.virginia.edu/~des5e/}}

\begin{abstract}
We determine when a matrix is similar to a partial isometry, refining a result of Halmos--McLaughlin. 
\end{abstract}

\keywords{Partial isometry; partially isometric matrix; similarity; Jordan form}
\thanks{The authors acknowledge the support of NSF grants DMS-1265973 and DMS-1201454.}	
	    \subjclass[2010]{15A21}

\maketitle

\section{Introduction}

A Hilbert space operator $V$ is a \emph{partial isometry} if the restriction of $V$
to $(\ker V)^{\perp}$ is isometric.  For a complex matrix, this means that all of its singular values are in $\{0,1\}$,
or in other words, the positive semidefinite factor in its
polar decomposition is an orthogonal projection.
These properties are not preserved by similarity; for example 
\begin{equation}\label{eq:TwoExamples}
    \begin{bmatrix}
        0 & \frac{ \sqrt{3} }{2} \\[2pt]
        0 & \frac{1}{2}
    \end{bmatrix}   
    \quad\text{and}\quad \minimatrix{0}{0}{0}{\frac{1}{2}} 
\end{equation}
are similar, since both matrices have the same Jordan canonical form.  
The first is a partial isometry since its nonzero columns are orthonormal, but the second is not.  Which matrices are similar to a partial isometry?

The basic features of partial isometries were laid out over fifty years ago \cite{Halmos, Erdelyi, Hearon}, and the similarity question is not new -- but most work has focused on the (still unresolved) infinite-dimensional case, e.g., \cite{Fialkow}.  
In the finite-dimensional case, the best result was a theorem of Halmos--McLaughlin stating that the characteristic polynomial of a nonunitary partial isometry can be any monic polynomial whose roots lie in the closed unit disk and include zero \cite[Theorem 3]{Halmos}.  
The referee pointed out that this can be deduced directly from the Weyl--Horn inequalities
\cite{Horna,Weyl}, which say that there exists an $n \times n$ matrix with prescribed singular values
$\sigma_1 \geq \sigma_2 \geq \cdots \sigma_n \geq 0$ and eigenvalues $\lambda_1,\lambda_2,\ldots,\lambda_n$,
indexed so that $|\lambda_1| \geq |\lambda_2| \geq \cdots \geq |\lambda_n|$, if and only if
\begin{equation*}
    \sigma_1 \sigma_2 \cdots \sigma_k \geq |\lambda_1 \lambda_2 \cdots \lambda_k| \qquad \text{for }k = 1,2,\ldots,n-1 \end{equation*}
and
\begin{equation*}
\sigma_1 \sigma_2 \cdots \sigma_n = |\lambda_1 \lambda_2 \cdots \lambda_n|.
\end{equation*}
For an $n \times n$ partial isometry of rank $r<n$, we have
\begin{equation*}
\sigma_1 = \sigma_2 = \cdots = \sigma_r = 1 \qquad \text{and}\qquad  \sigma_{r+1} = \cdots = \sigma_n = 0.
\end{equation*}
Hence
any $n$ points (with possible repetition) in the closed unit disk can be its eigenvalues,
so long as at least $n-r$ of them are zero.

Of course the characteristic polynomial is a similarity invariant for matrices, but not a complete one.  Here we go further and determine the possible Jordan forms of a partially isometric matrix.
One interesting feature of this question is that the property ``similar to a partial isometry'' does not pass to direct 
summands: the $1 \times 1$ matrix $[ \frac{1}{2} ]$, a direct summand of the second matrix in \eqref{eq:TwoExamples}, is not similar to a partial isometry.

In what follows, $\D$, $\D^-$, $\T$ denote the open unit disk, the closed unit disk, and the unit circle, respectively.  We write $\sim$ for similarity and $\cong$ for unitary similarity.
The spectrum of a matrix $A \in M_n$ is denoted $\sigma(A)$;
here $M_n$ denotes the set of $n \times n$ complex matrices.
We let $J_n(\lambda)$ denote the $n \times n$ Jordan matrix with eigenvalue $\lambda$, and we remind the reader that the nullity of $A - \lambda I$ is exactly the number of Jordan blocks for the eigenvalue $\lambda$.

    \begin{theorem}\label{Theorem:PI}
        $A \in M_n$ is similar to a partial isometry if and only if the following conditions hold:
        \begin{enumerate}\addtolength{\itemsep}{0.5\baselineskip}
            \item $\sigma(A) \subseteq \D^-$;
            \item if $\zeta \in \sigma(A) \cap \T$, then its algebraic and geometric multiplicities are equal;
            \item $\nullity A \geq \nullity (A - \lambda I)$ for each $\lambda \in \sigma(A) \cap \D$.
        \end{enumerate}
    \end{theorem}
    
It is classical (e.g., \cite[Lemma 1]{Hearon}) that partial isometries are exactly the operator solution set of the *-polynomial equation $xx^*x - x = 0.$  After proving Theorem \ref{Theorem:PI}, we make a few remarks on similarity orbits of such ``noncommutative *-varieties.''

\section{Proof of Theorem \ref{Theorem:PI}}
     Before proving Theorem \ref{Theorem:PI}, we prepare a few lemmas.  The first is a variation of the Halmos--McLaughlin result quoted earlier \cite[Theorem 3]{Halmos}, restricting to tuples in $\D$ and obtaining the extra condition (c).  

    \begin{lemma}\label{Lemma:Superdiagonal}
        For any $\xi_1,\xi_2,\ldots, \xi_{n-1} \in \D$, there
        exists an upper triangular partial isometry 
        $V = [ \vec{0}\,\, \vec{v}_2 \,\,\vec{v}_3 \,\,\ldots\,\, \vec{v}_{n-1} ] \in M_n$ such that 
        \begin{enumerate}\addtolength{\itemsep}{0.5\baselineskip}
            \item the diagonal of $V$ is $(0,\xi_1,\xi_2,\ldots, \xi_{n-1})$;
            \item the columns $\vec{v}_2, \vec{v}_3,\ldots, \vec{v}_{n-1} \in \C^n$ are orthonormal;
            \item the entries of $V$ on the first superdiagonal are all nonzero.
        \end{enumerate}
    \end{lemma}

    \begin{proof}
        We proceed by induction on $n$.  In the base case $n = 2$,
        \begin{equation*}
            V = \minimatrix{ 0}{ \sqrt{1 - | \xi_1|^2} }{0}{ \xi_1}
        \end{equation*}
        is a partial isometry with the desired properties.  Now suppose that the theorem has been
        proved for some $n\geq 2$.  Given $\xi_1,\xi_2,\ldots, \xi_{n} \in \D$,
        there exists a partial isometry 
        $W = [ \vec{0}\,\, \vec{w}_2 \,\,\vec{w}_3 \,\,\ldots\,\, \vec{w}_{n-1} ] \in M_n$ of the form
        \begin{equation*}
            W =
            \begin{bmatrix}
                0 & w_{1,2} & w_{1,3} & \cdots & w_{1,n} \\
                0 & \xi_{1} & w_{2,3} & \cdots & w_{2,n} \\
                0 & 0 & \xi_2 & \cdots & w_{3,n} \\[-2pt]
                0 & 0 & 0 & \ddots & \vdots \\
                0 & 0 & 0 & \cdots & \xi_{n-1} \\
            \end{bmatrix}
        \end{equation*}
        such that the columns $\vec{w}_2, \vec{w}_3,\ldots, \vec{w}_{n-1} \in \C^{n}$ are orthonormal.
        Now select a vector $\vec{w}_{n} \in \C^n$ of norm
        $\sqrt{1 - | \xi_n|^2}$ that is orthogonal to $\vecspan\{ \vec{w}_1, \vec{w}_2,\ldots, \vec{w}_{n-1}\}$.  
        Let
        \begin{equation*}
            V = \minimatrix{W}{ \vec{w}_n}{0}{ \xi_n} \in M_{n+1}.
        \end{equation*}
        Since the nonzero columns of $V$ are orthonormal,
        it follows that $V$ is an upper triangular partial isometry that satisfies (a) and (b).
        If $v_{n-1,n }=0$, then 
        \begin{equation*}
            V =\left[
            \begin{array}{c|ccccc}
                0 & v_{1,2} & \cdots & v_{1,n-2} & v_{1,n-1} & v_{1,n} \\
                0 & \xi_1 & \cdots & v_{2,n-2} & v_{2,n-1} & v_{2,n} \\[-2pt]
                0 & 0 & \ddots & \vdots & \vdots & \vdots \\
                0 & 0 & \cdots & \xi_{n-2} & v_{n-2,n-1} & v_{n-2,n} \\
                \hline
                0 & 0 & \cdots & 0 & \xi_{n-1} & 0 \\
                0 & 0 & \cdots & 0 & 0 & \xi_n 
            \end{array}
            \right],
        \end{equation*}
        so that the upper right $(n-2) \times (n-1)$ submatrix has orthogonal nonzero columns, which is impossible.
        Thus $v_{n-1,n} \neq 0$, which completes the induction.
    \end{proof}

\begin{lemma}\label{Lemma:JCFB}
    If $T \in M_n$ is upper triangular with $\sigma(T) = \{ \lambda\}$,
    and the entries on the first superdiagonal of $T$ are all nonzero, then $T \sim J_n(\lambda)$.
\end{lemma}

\begin{proof}
    By hypothesis, $\rank (T - \lambda I) = n-1$.  Thus the Jordan canonical form of $T$
    consists of the single block $J_n(\lambda)$.
\end{proof}

        The following lemma is \cite[Theorem 2.4.6.1]{HJ}:
	
	\begin{lemma}\label{Lemma:Jordan}
		Suppose that $T = [T_{ij}]_{i,j}^{d}\in M_n$ is block upper triangular,
		and each $T_{ii} \in M_{n_{i}}(\C)$ is upper triangular with all diagonal entries equal to $\lambda_i$.
		If $\lambda_i \neq \lambda_j$
		for $i \neq j$, then $T \sim T_{11} \oplus T_{22} \oplus \cdots \oplus T_{dd}$.
	\end{lemma}

\begin{proof}[Proof of Theorem \ref{Theorem:PI}]
($\Rightarrow$) Conditions (a), (b), and (c) are invariant under similarity, so it suffices to show that they are satisfied by any partial isometry $V$.  Since $V$ is a contraction, it follows that $\sigma(V) \subseteq \D^-$.  This is (a).

Now suppose that 
$\zeta \in \sigma(V) \cap \T$ has algebraic multiplicity $m$.  By Schur's Theorem on unitary
triangularization \cite[Theorem 2.3.1]{HJ},
we may assume that $V$ has the form
\begin{equation*}
V = \minimatrix{V_{11}}{V_{12}}{0}{V_{22}}, 
\end{equation*}
in which $V_{11} \in M_m$ and $V_{22} \in M_{n-m}$ are upper triangular, $\sigma(V_{11}) = \{\zeta\}$,
and $\zeta \notin \sigma(V_{22})$.  This entails that the diagonal entries of $V_{11}$ are all $\zeta$.  
Since $V$ is a contraction, each of its columns and rows has norm at most one, forcing $V_{11} = \zeta I_m$ and $V_{12} = 0$.  Then
$V = \zeta I_m \oplus V_{22}$.  Since $\zeta \notin \sigma(V_{22})$, it follows that
the algebraic and geometric multiplicities of $\zeta$, as an eigenvalue of $V$, are equal.
Thus (b) holds.

Finally, let $r = \rank V$ and write $V=UP$ for some unitary $U$ and projection $P$ of rank $r$.  If $\lambda \in \D$, then $U - \lambda I$ is nonsingular and hence
\begin{align*}
    n &= \rank( U - \lambda I)  =\rank[ (UP - \lambda I) + U(I-P)]  \\ &\leq \rank(UP - \lambda I) + \rank(U(I-P)) = \rank(V- \lambda I) + n-r.
\end{align*}
Thus $\rank V = r \leq \rank (V - \lambda I)$, or in other words $\nullity V \geq \nullity (V - \lambda I)$.  This proves (c).
\medskip

\noindent ($\Leftarrow$)
Suppose that $A \in M_n$ satisfies (a), (b), and (c).
If $\sigma(A) \cap \T = \{ \zeta_1, \zeta_2,\ldots, \zeta_s\}$,
in which $\zeta_1,\zeta_2,\ldots,\zeta_s$ are distinct and have algebraic multiplicities
$m_1,m_2,\ldots,m_s$, respectively, then
\begin{equation*}
    A \sim A' \oplus \zeta_1 I_{m_1} \oplus \zeta_2 I_{m_2} \oplus \cdots \oplus \zeta_s I_{m_s}, 
\end{equation*}
in which $\zeta_1,\zeta_2,\ldots, \zeta_s \notin \sigma(A')$ (it is possible that $A'$ is vacuous, in which case $A$ is similar
to a unitary matrix).  
Since each summand $\zeta_i I_{m_i}$
is an isometry, $A$ is similar to a partial isometry if $A'$ is.  We therefore assume that
$\sigma(A) \cap \T = \varnothing$.  This ensures that $A$ has an eigenvalue in $\D$,
so (c) implies that $0 \in \sigma(A)$.

Let $m = \nullity A$, which equals the number of Jordan blocks for the eigenvalue $0$
in the Jordan canonical form of $A$.  Condition (c) ensures that the Jordan canonical form
of $A$ has at most $m$ Jordan blocks corresponding to any nonzero eigenvalue of $A$.
It therefore suffices to show that any matrix of the form
\begin{equation}\label{eq:ASDSMFB}
B = J_{n_0}(0) \oplus \bigoplus_{i=1}^d J_{n_i}(\lambda_i),\qquad d \geq 0, \quad 0 < n_i \leq n,
\end{equation}
in which 
$\lambda_1,\lambda_2,\ldots, \lambda_d \in \D \backslash\{0\}$ are distinct, is similar to a partial
isometry.  Indeed, $A$ is similar to a direct sum of matrices of the form \eqref{eq:ASDSMFB}.

Lemma \ref{Lemma:Superdiagonal} ensures that there exists
a partial isometry $V \in M_n$ 
with nonzero entries on the first superdiagonal,  
whose diagonal entries are (in order)
\begin{equation*}
\underbrace{ 0,0,\ldots,0}_{\text{$n_0$ times}}, \underbrace{ \lambda_1,\lambda_1,\ldots,\lambda_1}_{\text{$n_1$ times}}, 
\underbrace{\lambda_2,\lambda_2,\ldots,\lambda_2}_{\text{$n_2$ times}},\ldots,
\underbrace{\lambda_d,\lambda_d,\ldots,\lambda_d}_{\text{$n_d$ times}}.
\end{equation*}
Partition $V$ conformally with $B$, so that $V_{i,j} \in M_{n_i \times n_j}$.
Lemma \ref{Lemma:JCFB} ensures that $V_{i,i} \sim J_{n_i}(\lambda_i)$, so
$V \sim B$ by Lemma \ref{Lemma:Jordan}.  
\end{proof}

\noindent\textbf{Remark}.
A general result about spectral matrices implies 
that a partial isometry satisfies (b) of Theorem \ref{Theorem:PI}.
A matrix is \emph{spectral} if its spectral radius equals its numerical radius.
A partial isometry $V \in M_n$ that has an eigenvalue $\zeta$ of unit modulus 
has spectral radius and numerical radius $1$ and is therefore spectral.
If the multiplicity of $\zeta$ is $m$, then \cite[Theorem 1]{Goldberg} ensures that $V \cong \zeta I_m \oplus V_{22}$,
in which $V_{22} \in M_{n-m}$ is an upper triangular partial isometry that does not have $\zeta$ as an eigenvalue.
See also \cite[Section 1.5, Problems 24 \& 27]{Topics}.

\section{Noncommutative *-varieties}

The common operator solution set of a finite collection of *-polynomial equations in some number of variables is a \textit{noncommutative *-variety}; when the polynomials do not involve adjoints, it is a \textit{noncommutative variety}.  (This term is in general use, although the precise definition varies from paper to paper.)   Noncommutative varieties are invariant under similarity, but noncommutative *-varieties typically are invariant only under unitary similarity.  Theorem \ref{Theorem:PI} describes the similarity orbit of the matricial part of the noncommutative *-variety determined by the polynomial $x x^* x - x$.

Many other sets can be described as *-varieties.  Hermitian operators correspond to $x-x^*$, normal operators to $x^*x-xx^*$.  It is easy to describe the similarity orbit of the normal matrices: since these are unitarily diagonalizable, it means exactly that for each eigenvalue, the algebraic and geometric multiplicities are equal.  For Hermitian matrices, one adds the requirement that all eigenvalues are real.  Here is another example that is not too different from Theorem \ref{Theorem:PI}, making one small change to the *-polynomial determining partial isometries.

\begin{proposition}
A matrix $T$ is similar to a solution of $x x^*x - x^2=0$ if and only if
\begin{enumerate}\addtolength{\itemsep}{0.5\baselineskip}
\item the spectrum of $T$ is contained in $[0,1]$;
\item for each eigenvalue, the geometric and algebraic multiplicities are equal;
\item the nullity of $T$ is at least as great as the sum of the nullities of $T-cI$ for all $c \in (0,1)$.
\end{enumerate}
\end{proposition}

\begin{proof}
By \cite[Theorem 8]{RW}, the noncommutative *-variety determined by $xx^*x-x^2$ consists exactly of products of pairs of orthogonal
projections.  We proceed by determining the Jordan form of a product of two orthogonal projections $P,Q$.

From Halmos' two projections theorem (\!\cite[Theorem 2]{2sub}, see also \cite[Theorem 1.1 and Corollary 2.2]{Bottcher}), we deduce that 
\begin{align*}
&P \cong I_{d_1} \oplus I_{d_2} \oplus 0_{d_3} \oplus 0_{d_4} \oplus \bigoplus_{j=1}^n \left[\begin{smallmatrix} 1 & 0 \\ 0 & 0 \end{smallmatrix} \right] , \\ &Q \cong I_{d_1} \oplus 0_{d_2} \oplus I_{d_3} \oplus 0_{d_4} \oplus \bigoplus_{j=1}^n\left[\begin{smallmatrix} c_j^2 & c_j \sqrt{1 - c_j^2} \\ c_j \sqrt{1 - c_j^2} & 1 - c_j^2 \end{smallmatrix}\right],
\end{align*}
where $c_j \in (0,1)$, $d_1,d_2,d_3,d_4 \geq 0$, and $n \geq 0$.  Thus
$$PQ \cong I_{d_1} \oplus 0_{d_2} \oplus 0_{d_3} \oplus 0_{d_4} \oplus \bigoplus_{j=1}^n \left[\begin{smallmatrix} c_j^2 & c_j \sqrt{1 - c_j^2} \\ 0 & 0 \end{smallmatrix} \right],$$
which is similar to
\begin{equation}\label{eq:LJCF}
I_{d_1} \oplus 0_{d_2} \oplus 0_{d_3} \oplus 0_{d_4} \oplus \bigoplus_{j=1}^n \left[\begin{smallmatrix} c_j^2 & 0 \\ 0 & 0 \end{smallmatrix}\right].
\end{equation}
In a basis of eigenvectors of $PQ$, each eigenvector for $c_j^2 \in (0,1)$ is paired with an eigenvector for $0$.  This yields (a), (b), and (c).

Conversely, any Jordan form satisfying (a), (b), and (c) can be written as in \eqref{eq:LJCF}
and is similar to the product of orthogonal projections $P$ and $Q$ as above.
\end{proof}

\noindent\textbf{Remark.} 
Taking the similarity orbit is one way to change a noncommutative *-variety.  Another way is to replace the defining equalities with similarities.  This generally produces a different set, not necessarily smaller or larger than the similarity orbit, and also not canonical because it requires a choice about how to write the original equations.  For instance, the same *-variety is determined by $x - x^* = 0$ or $x = x^*$; it could become $x - x^* \sim 0$ (which is still just the set of Hermitian operators) or $x \sim x^*$. The latter set has been much studied,
sometimes under the name \emph{generalized Hermitian operators},
and is larger than the similarity orbit of the Hermitian operators.  
Here is one comparison: a matrix is similar to its adjoint if and only if it is the product of two Hermitian matrices, while it is similar to a Hermitian matrix if and only if it is the product of a positive definite and a Hermitian \cite{W,RW}.  
See \cite{SM} for a modern treatment of generalized Hermitian operators.

Applying this variation to our Theorem \ref{Theorem:PI}, let $\SSS$ be the set of matrices satisfying $TT^*T \sim T$.  Unlike the other sets considered in this paper, it appears that
$\SSS$ might not have an alternative easy description.  It contains $[\begin{smallmatrix} 1 & 1 \\ 0 & -1 \end{smallmatrix}]$ (not a partial isometry), but it does not contain $[\begin{smallmatrix} \frac{1}{\sqrt{2}} & 1 \\ 0 & 0 \end{smallmatrix}]$, which is similar to the partial isometry $[\begin{smallmatrix} \frac{1}{\sqrt{2}} & \frac{1}{\sqrt{2}} \\ 0 & 0 \end{smallmatrix}]$.  Thus $\SSS$ is distinct from the set of partially isometric matrices and its similarity orbit, and in fact $\SSS$ is not closed under similarity.

\bibliography{MSPI}{}

\begin{thebibliography}{10}

\bibitem{Bottcher}
A.~B\"ottcher and I.~M. Spitkovsky.
\newblock A gentle guide to the basics of two projections theory.
\newblock {\em Linear Algebra Appl.}, 432(6):1412--1459, 2010.

\bibitem{Erdelyi}
I.~Erd{\'e}lyi.
\newblock On partial isometries in finite-dimensional {E}uclidean spaces.
\newblock {\em SIAM J. Appl. Math.}, 14:453--467, 1966.

\bibitem{Fialkow}
L.~A. Fialkow.
\newblock Which operators are similar to partial isometries?
\newblock {\em Proc. Amer. Math. Soc.}, 56:140--144, 1976.

\bibitem{Goldberg}
M.~Goldberg, E.~Tadmor, and G.~Zwas.
\newblock The numerical radius and spectral matrices.
\newblock {\em Linear and Multilinear Algebra}, 2:317--326, 1975.

\bibitem{2sub}
P.~R. Halmos.
\newblock Two subspaces.
\newblock {\em Trans. Amer. Math. Soc.}, 144:381--389, 1969.

\bibitem{Halmos}
P.~R. Halmos and J.~E. McLaughlin.
\newblock Partial isometries.
\newblock {\em Pacific J. Math.}, 13:585--596, 1963.

\bibitem{Hearon}
J.~Z. Hearon.
\newblock Partially isometric matrices.
\newblock {\em J. Res. Nat. Bur. Standards Sect. B}, 71B:225--228, 1967.

\bibitem{Horna}
A.~Horn.
\newblock On the eigenvalues of a matrix with prescribed singular values.
\newblock {\em Proc. Amer. Math. Soc.}, 5:4--7, 1954.

\bibitem{Topics}
R.~A. Horn and C.~R. Johnson.
\newblock {\em Topics in matrix analysis}.
\newblock Cambridge University Press, Cambridge, 1994.
\newblock Corrected reprint of the 1991 original.

\bibitem{HJ}
R.~A. Horn and C.~R. Johnson.
\newblock {\em Matrix analysis}.
\newblock Cambridge University Press, Cambridge, second edition, 2013.

\bibitem{RW}
H.~Radjavi and J.~P. Williams.
\newblock Products of self-adjoint operators.
\newblock {\em Michigan Math. J.}, 16:177--185, 1969.

\bibitem{SM}
S.~Sun and X.~Ma.
\newblock Generalized {H}ermitian operators.
\newblock {\em Linear Algebra Appl.}, 433(4):737--749, 2010.

\bibitem{Weyl}
H.~Weyl.
\newblock Inequalities between the two kinds of eigenvalues of a linear
  transformation.
\newblock {\em Proc. Nat. Acad. Sci. U. S. A.}, 35:408--411, 1949.

\bibitem{W}
J.~P. Williams.
\newblock Operators similar to their adjoints.
\newblock {\em Proc. Amer. Math. Soc.}, 20:121--123, 1969.

\end{thebibliography}
\bibliographystyle{plain}

\end{document}